\documentclass[a4paper, 12pt]{amsart}
\usepackage[utf8]{inputenc}
\usepackage{xcolor}

\newtheorem{theorem}[equation]{Theorem}
\newtheorem{lemma}[equation]{Lemma}
\newtheorem{definition}[equation]{Definition}
\newtheorem{proposition}[equation]{Proposition}
\newtheorem{corollary}[equation]{Corollary}

\theoremstyle{remark}
\newtheorem{example}[equation]{Example}
\newtheorem{remark}[equation]{Remark}
\newtheorem*{notation}{Notation}

\numberwithin{equation}{section}

\newcommand{\Gg}{\mathcal{G}} 
\newcommand{\rel}{\mathcal{R}} 
\newcommand{\triple}{(\Gg,\rel_N,\rel_S)} 
\newcommand{\cn}{\mathbb{C}} 
\newcommand{\rn}{\mathbb{R}} 
\newcommand{\nn}{\mathbb{N}} 
\newcommand{\falg}{\mathcal{B}_{\Gg}} 
\newcommand{\seminormset}{\mathcal{P}_{\triple}} 
\newcommand{\snideal}{\mathcal{N_{\triple}}} 
\newcommand{\scj}{\subseteq} 
\newcommand{\cuntz}[1]{\mathcal{O}_{#1}} 
\newcommand{\elu}{\widetilde{\mathcal{O}}_A} 
\newcommand{\fr}{\mathbb{F}}
\newcommand{\newfd}[5]{\begin{array}{cccc}{#1}: &{#2} &\longrightarrow &{#3} \\& {#4} &\longmapsto &{#5}\end{array}}
\newcommand{\circulo}{\mathbb{T}}

\newcommand{\ug}{\ugg{E}}
\newcommand{\ugg}[1]{\mathcal{#1}}
\newcommand{\ugquadg}[1]{(#1^0,\ugg{#1}^1,s,r)}
\newcommand{\ugquad}{\ugquadg{E}}

\DeclareMathOperator{\spanop}{span}

\DeclareFontFamily{U}{matha}{\hyphenchar\font45}
\DeclareFontShape{U}{matha}{m}{n}{
      <5> <6> <7> <8> <9> <10> gen * matha
      <10.95> matha10 <12> <14.4> <17.28> <20.74> <24.88> matha12
      }{}
\DeclareSymbolFont{matha}{U}{matha}{m}{n}
\DeclareFontSubstitution{U}{matha}{m}{n}

\DeclareFontFamily{U}{mathx}{\hyphenchar\font45}
\DeclareFontShape{U}{mathx}{m}{n}{
      <5> <6> <7> <8> <9> <10>
      <10.95> <12> <14.4> <17.28> <20.74> <24.88>
      mathx10
      }{}
\DeclareSymbolFont{mathx}{U}{mathx}{m}{n}
\DeclareFontSubstitution{U}{mathx}{m}{n}

\DeclareMathDelimiter{\vvvert}{0}{matha}{"7E}{mathx}{"17}

\makeatletter
\@namedef{subjclassname@2020}{%
  \textup{2020} Mathematics Subject Classification}
\makeatother

\title{Infinite sum relations on universal C*-algebras}
\author{Giuliano Boava}

\author{Gilles G. de Castro}
\thanks{The second author was partially supported by Capes-PrInt Brazil grant number 88881.310538/2018-01.}

\address{Departamento de Matemática, Universidade Federal de Santa Catarina, 88040-970 Florianópolis SC, Brazil.}
\email{g.boava@ufsc.br \\ gilles.castro@ufsc.br}


\begin{document}

\maketitle

\begin{abstract}
We extend the usual theory of universal C*-algebras from generators and relations in order to allow some relations to be described using the strong operator topology. In particular, we can allow some infinite sum relations. We prove a universal property for the algebras we define and we show how the Cuntz algebra of infinite isometries as well as the Exel-Laca algebras can be described using infinite sum relations. Finally, we give some sufficient conditions for when a C*-algebra generated by projections and partial isometries is a universal C*-algebra using only norm relations, in case one still wants to avoid using relations with respect to the strong operator topology.
\end{abstract}

\section{Introduction}

One tool used to define a C*-algebra is to consider a certain family of operators on a Hilbert space $H$ and then consider the C*-subalgebra of $B(H)$ generated by this family, that is, we define a C*-algebra represented on $H$. If we want the family to satisfy certain conditions, such as reflect the properties of another mathematical object, one can ask to what extent the algebra depends of the representation or if it can be uniquely defined.

Looking at the history of C*-algebras generated by partial isometries - including projections, isometries and unitaries - we see examples of such uniqueness theorems: Coburn's work on the C*-algebra generated by one isometry \cite{MR213906}, Cuntz's result on the simplicity of certain algebras generated by isometries with mutually orthogonal ranges \cite{MR467330}, and the uniqueness theorem of Cuntz and Krieger for the C*-algebra defined from a square matrix of zeros and ones \cite{MR561974}.

Instead of proving a uniqueness theorem for each class of algebra, we can try to define a C*-algebra in terms of a universal property. This was done by Blackadar in \cite{MR813640} by defining a universal C*-algebra in terms of generators and relations. Although he hinted that we could use any relation for operators on a Hilbert space or for elements on a C*-algebra, he only developed the theory for what we call norm relations in this paper. His first example said that any C*-algebra is universal with respect to all elements as generator and all *-algebraic relations in the algebra, however this is not very helpful when defining a C*-algebra. Indeed, if we define a C*-algebra using operators as explained above, we would like to use a set of generators corresponding the operators used to define the C*-algebra. For the algebras of the previous paragraph, Blackadar showed how to do that, except for the Cuntz-Krieger algebras associated with infinite matrices and implicitly the Cuntz algebra $\cuntz{\infty}$, since we need to deal with infinite sums converging in the strong operator topology and these cannot be directly described in terms of norm relations.

Going back to the history of C*-algebras generated by partial isometries after Blackadar's paper, we get the impression that using a relation involving an infinite sum is a forbidden practice. For example, when studying Cuntz-Krieger algebras of infinite matrices, Exel and Laca in \cite{MR1703078} state that the theory developed in \cite{MR813640} ``breaks down when the relations involve the strong
topology'' so they found norm relations in order to define a universal C*-algebra for infinite matrices of zeros and ones. Other examples are considering at first only row-finite graphs for the C*-algebra of graph \cite{MR1626528} and asking that an integral domain has the property that the ideal generated by a non-zero element has finite index \cite{MR2732050}.

The main goal of this paper is to extend Blackadar's construction in order to define a universal C*-algebra that allows for relations using the strong operator topology, more specifically we allow as a relation a net of operators converging to zero with respect to the strong operator topology. This is done in Section \ref{sec:definition.universal.algebra}. For the universal property, in order to define a *-homomorphism from our algebra to a C*-algebra $A$, we need a faithful representation of $A$ on a Hilbert space so that we can have access to the strong operator topology (Theorem \ref{thm:universal.property}). Another difficulty added by allowing relations involving the strong operator topology is that not all representations of the universal C*-algebra satisfies the relations, as shown in Example \ref{sec:cuntz}, however we prove that there exists a faithful representation of the algebra that satisfies all relations (Theorem \ref{thm:faithful.rep}).

In Section \ref{sec:Examples}, we work out in details how we can indeed allow infinite sum relations in order to define the Cuntz algebra and Exel-Laca algebras. In particular for Exel-Laca algebras, the relations we use are exactly the infinite sum relations Exel and Laca avoided in \cite{MR1703078}.

Finally, despite the main goal of this paper being using relations with respect to the strong operator topology in order to define a universal version of a given C*-algebra, in Section \ref{sec:algebra.partial.isometry}, we deal only with norm relations. We show that under some conditions a C*-algebra generated by projections and partial isometries, independently on how it was defined, can be described using as generators a family corresponding to the projections and partial isometries and only using norm relations (Theorems \ref{lem:commuting.projections.universality} and \ref{thm:alg.gen.proj.part.iso}), in case one still wants to avoid relations using the strong operator topology.

The other result of Section \ref{sec:algebra.partial.isometry}, namely Corollary \ref{cor:subalgebra.isomorphism}, gives a strategy for finding a ``correct'' set of norm relations describing a certain C*-algebra. As mentioned above, sometimes a C*-algebra generated by partial isometries (and projections) is originally defined represented in a Hilbert space. To define a universal version of this algebra, we can start by using all algebraic relations that can be written using generators associated to each partial isometry (including the projections). The result of Corollary \ref{cor:subalgebra.isomorphism} says that at least we have a *-isomorphism between a distinguished commutative C*-subalgebra generated by projections of the universal and represented versions. One way of using this conclusion is in avoiding relations with respect to the strong operator topology involving only a distinguished subset of mutually commuting projections by replacing them with all algebraic relations. However, one usually wants to choose a smaller list of relations and expects that this list encodes all other relations. Another way of using the conclusion of Corollary \ref{cor:subalgebra.isomorphism} is that we can take the spectrum of this commutative C*-subalgebra generated by the projections to define, for example, a partial action or a groupoid and then a C*-algebra from them. If the new construction and the universal C*-algebra of the restricted list of relations coincide, it means that we arrived at the ``correct'' set of relations. Implicitly, this was done in \cite{MR1703078}, when they showed that their version of the Cuntz-Krieger algebras for infinite matrices is isomorphic to a partial crossed product.

\section{Strong operator topology relations on universal C*-algebras}\label{sec:definition.universal.algebra}

In this section, we adapt Blackadar's construction \cite{MR813640} to include relations that can be used to describe how some infinite sums should behave with respect to the strong operator topology (SOT) when a C*-algebra satisfying these relations is represented on a Hilbert space.

Let $\Gg$ be a non-empty set and $\Gg^*$ be a disjoint copy of $\Gg$, where each element of $\Gg^*$ is denoted by $g^*$ for some $g\in\Gg$. Let $\falg:=\cn\langle\Gg\cup\Gg^*\rangle$ be the free $\cn$-algebra generated by $\Gg\cup\Gg^*$ with the standard involution, that is, $\falg$ is the free *-algebra over $\cn$ generated by $\Gg$.

\begin{definition}\label{def:relations}
Let $\Gg$ and $\falg$ be as above. We define a \emph{norm relation} as a pair $(x,\eta)\in\falg\times\rn_+$, and a \emph{SOT relation} as a net $(x_i)_{i\in I}$ on $\falg$. A \emph{generating triple} is a triple $\triple$, where $\rel_N$ is a set of norm relations and $\rel_S$ is a set of SOT relations. 
\end{definition}

The idea is that a norm relation $(x,\eta)$ is interpreted as the inequality $\|x\|\leq\eta$ when $x$ is represented on a C*-algebra, and a SOT relation $(x_i)_{i\in I}$ interpreted as $s-\lim_{i\to\infty}x_i=0$ (limit on the strong operator topology), when the net is represented as operators on a Hilbert space. We use this as motivation for the following definition.

\begin{definition}
Let $\triple$ be a generating triple and $H$ a Hilbert space. We say that a *-algebra homomorphism $\rho:\falg\to B(H)$ is a \emph{representation of $\triple$} if $\|\rho(x)\|\leq\eta$ for every $(x,\eta)\in\rel_N$ and $s-\lim_{i\to\infty}\rho(x_i)=0$ for every $(x_i)_{i\in I}\in\rel_S$.
\end{definition}

\begin{remark}\label{rem:hom.from.generators}
Let $\triple$ be a generating triple. Since $\falg$ is the free *-algebra over $\cn$ generated by $\Gg$, in order to describe a *-homomorphism $\varphi:\falg\to A$, where $A$ is a *-algebra, it is sufficient to know the values $\varphi(g)$ for the elements $g\in\Gg$. In particular, if $\{T_g\}_{g\in \Gg}$ is a family of operators in $B(H)$, we say that this family forms a representation of $\triple$ if the map $\rho:\falg\to B(H)$ defined by $\rho(g)=T_g$ is a representation of $\triple$.
\end{remark}

In order to build a universal C*-algebra from a generating triple, we need to somehow control the representations of the triple. This is done with the notion of admissibility as in \cite[Definition 1.1]{MR813640}.

\begin{definition}
We say that the generating triple $\triple$ is \emph{admissible} if for every family $\{\rho_{\alpha}:\falg\to B(H_\alpha)\}_{\alpha}$ of representations of $\triple$, the map
\begin{equation}\label{eq:sum.representations}
\newfd{\bigoplus_{\alpha}\rho_{\alpha}}{\falg}{B\left(\bigoplus_\alpha H_\alpha\right)}{x}{\bigoplus_{\alpha}\rho_{\alpha}(x)}
\end{equation}
is a well-defined representation of $\triple$.
\end{definition}

The following result is useful when checking admissibility of a triple.

\begin{proposition}\label{prop:admissibility}
Let $\triple$ be a generating triple. If for every $g\in\Gg$ there exists $\eta_g\in\rn_+$ and for every net $\mathbf{x}=(x_i)_{i\in I}\in\rel_S$ there exists $\eta_{\mathbf{x}}\in\rn_+$ such that for every representation $\rho$ of $\triple$, we have that $\|\rho(g)\|\leq \eta_g$ and $\|\rho(x_i)\|\leq \eta_{\mathbf{x}}$ for every $i\in I$, then $\triple$ is admissible.
\end{proposition}

\begin{proof}
By using the properties of the norm on a C*-algebra, we can easily check that for every $x\in\falg$, there exists $\eta_x\in\rn_+$ such that $\|\rho(x)\|\leq\eta_x$ for all representation $\rho$ of $\triple$. This implies that given a family $\{\rho_{\alpha}:\falg\to B(H_\alpha)\}_{\alpha\in \Omega}$ of representations of $\triple$, the map given by (\ref{eq:sum.representations}) is well-defined. It remains to show that $\bigoplus_{\alpha\in\Omega}\rho_{\alpha}$ is a representation of $\triple$.

For $(x,\eta)\in\rel_N$, we have that
\[\left\|\bigoplus_{\alpha\in\Omega}\rho_{\alpha}(x)\right\|=\sup_{\alpha\in\Omega}\|\rho_\alpha(x)\|\leq\eta.\]

Suppose now that $\mathbf{x}=(x_i)_{i\in I}\in\rel_S$ and take $\xi=(\xi_\alpha)_{\alpha\in\Omega}\in\bigoplus_{\alpha\in\Omega} H_\alpha$. Given $\varepsilon>0$, there exists a finite set $F\scj\Omega$ such that $\sum_{\alpha\in\Omega\setminus F}\|\xi_\alpha\|^2<\frac{\varepsilon}{2(\eta_{\mathbf{x}}^2+1)}$. Also, since $\lim_{i\to\infty}\rho_{\alpha}(x_i)(\xi_\alpha)=0$ for every $\alpha\in\Omega$, using that $I$ is a directed set, we can find $i_0\in I$ such that for every $i\geq i_0$ and every $\alpha\in F$, we have that $\|\rho_\alpha(x_i)(\xi_\alpha)\|^2<\frac{\varepsilon}{2(|F|+1)}$, where $|F|$ is the cardinality of $F$. Hence, for $i\geq i_0$,
\[\left\|\bigoplus_{\alpha\in\Omega}\rho_{\alpha}(x_i)(\xi)\right\|^2=\sum_{\alpha\in\Omega}\|\rho_{\alpha}(x_i)(\xi_\alpha)\|^2\leq \sum_{\alpha\in F}\|\rho_{\alpha}(x_i)(\xi_\alpha)\|^2+\sum_{\alpha\in\Omega\setminus F}\|\rho_{\alpha}(x_i)\|^2\|\xi_\alpha\|^2< \]
\[|F|\frac{\varepsilon}{2(|F|+1)}+\eta_{\mathbf{x}}^2\frac{\varepsilon}{2(\eta_{\mathbf{x}}^2+1)}<\varepsilon.\]
Therefore, $\lim_{i\to\infty}\bigoplus_{\alpha\in\Omega}\rho_{\alpha}(x_i)(\xi)=0$, and hence $s-\lim_{i\to\infty} \bigoplus_{\alpha\in\Omega}\rho_{\alpha}(x_i)=0$, since $\xi$ was arbitrary.
\end{proof}

The following definition is motivated by the previous proposition.
\begin{definition}
Let $\triple$ be a generating triple. We say that a relation $\mathbf{x}=(x_i)_{i\in I}\in\rel_S$ is bounded by $\eta\in\rn_+$ if for every representation $\rho$ of $\triple$, we have that $\|\rho(x_i)\|\leq \eta$ for every $i\in I$.
\end{definition}

The next goal is to define a C*-seminorm on $\falg$ in order to build a C*-algebra from a generating triple $\triple$. For that, let $\seminormset$ be the set of all C*-seminorms $p:\falg\to\rn$ such that $p(x)=\|\rho(x)\|$ for all $x\in\falg$, where $\rho$ is a representation of $\triple$.

\begin{proposition}
Let $\triple$ be an admissible generating triple. Then, \[\sup_{p\in\seminormset}p(x)<\infty\] for all $x\in\falg$.
\end{proposition}

\begin{proof}
Suppose that $x\in\falg$ is such that
\[\sup_{p\in\seminormset}p(x)=\infty.\]
Then, for each $n\in\nn$, there exists a representation $\rho_n$ of $\triple$ such that $\|\rho_n(x)\|\geq n$. This would imply that $\bigoplus_{n\in\nn}\rho_n(x)$ is an unbounded operator, contradicting the admissibility of $\triple$.
\end{proof}

Let $\triple$ be an admissible generating triple. For each $x\in\falg$, we define $\vvvert x\vvvert=\sup_{p\in\seminormset}p(x)$ and $\snideal=\{x\in\falg\mid \vvvert x\vvvert =0\}$. It's well known that $\vvvert\cdot\vvvert:\falg\to\rn$ is a C*-seminorm, $\snideal$ is a self adjoint ideal of $\falg$ and $\|\cdot\|:\falg/\snideal\to\rn$ given by
\[\|x+\snideal\|=\inf_{y\in\snideal}\vvvert x+y \vvvert\]
is a norm.
\begin{definition}
Let $\triple$ be an admissible generating triple. We define the C*-algebra generated by $\triple$ as the completion of $\falg/\snideal$ with respect to the norm $\|\cdot\|$ above and we denote it by $C^*\triple$. We also denote by $i:\falg\to C^*\triple$ the canonical *-homomorphism.
\end{definition}

\begin{remark}
If there are no SOT relations, that is $\rel_S=\emptyset$, then $C^*\triple=C^*(\Gg,\rel_N)$, where $C^*(\Gg,\rel_N)$ is the C*-algebra considered by Blackadar in \cite{MR813640}.
\end{remark}

We are now ready to state and prove a universal property for $C^*\triple$.

\begin{theorem}\label{thm:universal.property}
Let $\triple$ be an admissible generating triple and let $A$ be a C*-algebra. If $\varphi:\falg\to A$ is a *-homomorphism such that there exists a faithful representation $\pi:A\to B(H)$ of $A$ on a Hilbert space $H$ for which $\pi\circ\varphi$ is a representation of $\triple$, then there is a unique *-homomorphism $\widetilde{\varphi}:C^*\triple\to A$ such that $\varphi=\widetilde{\varphi}\circ i$.
\end{theorem}

\begin{proof}
We prove that the map $\widetilde{\varphi}:C^*\triple\to A$ given by $\widetilde{\varphi}(x+\snideal)=\varphi(x)$ is well-defined. For that, take $x\in\snideal$ arbitrary. By the definition of $\snideal$, we have that $\|\pi(\varphi(x))\|=0$, which implies that $\pi(\varphi(x))=0$. Since $\pi$ is faithful, we have that $\varphi(x)=0$ showing that $\widetilde{\varphi}$ is indeed well-defined.

Clearly $\varphi=\widetilde{\varphi}\circ i$ and $\widetilde{\varphi}$ is the unique *-homomorphism from $C^*\triple$ to $A$ satisfying this property.
\end{proof}

\begin{remark}
Generally, relations defining C*-algebras are given in terms of a equality. If all products and sums are finite, we present the relation $(x,0)\in\rel_N$ as $x=0$ (or some variation of the expression). For relations involving infinite sums or products, we have to somehow transform it as net converging to zero on the strong operator topology. We will show how to this in Section \ref{sec:Examples}. 
\end{remark}

We now prove that the canonical *-homomorphism $i:\falg\to C^*\triple$ satisfies the hypothesis of Theorem \ref{thm:universal.property}. First, we need a lemma.

\begin{lemma}\label{lem:good.rep}
Let $\triple$ be an admissible generating triple. For every $x\in\falg\setminus\snideal$, there exists a representation $\rho$ of $\triple$ such that $\|\rho(x)\|=\vvvert x\vvvert$.
\end{lemma}

\begin{proof}
Since $x\notin\snideal$, we have that $\vvvert x\vvvert >0$. Then, for each positive integer $n$, there exists $\rho_n$ representation of $\triple$ such that $0\leq\vvvert x \vvvert-\|\rho_n(x)\| < 1/n$. Since the triple is admissible, $\rho=\bigoplus_n \rho_n$ is a representation of $\triple$. For this $\rho$, we have that $\|\rho(x)\|=\sup_n\|\rho_n(x)\|=\vvvert x\vvvert$.
\end{proof}

\begin{theorem}\label{thm:faithful.rep}
Let $\triple$ be an admissible generating triple. Then, there exists a faithful representation $\pi_u:C^*\triple\to B(H_u)$ on a Hilbert space $H_u$ such that $\pi_u\circ i$ is a representation of $\triple$.
\end{theorem}

\begin{proof}
For each $x\in\falg\setminus\snideal$, we take a representation $\rho_x$ of $\triple$ such that $\|\rho_x\|=\vvvert x\vvvert$ as in Lemma \ref{lem:good.rep}. By Theorem \ref{thm:universal.property}, there exists a representation $\pi_x$ of $C^*\triple$ such that $\rho_x=\pi_x\circ i$. Define $\pi_u=\bigoplus_{x\in\falg\setminus\snideal}\pi_x$.

Notice that $\pi_u\circ i=\bigoplus_x\pi_x\circ i=\bigoplus_x \rho_x$, which is a representation of $\triple$, since the triple is admissible.

It remains to show that $\pi_u$ is faithful. For $a\in C^*\triple\setminus\{0\}$, let $x\in\falg\setminus\snideal$ be such that $\|a-i(x)\|\leq\|a\|/4$. An easy computation with the definition of $\| i(x)\|$ shows that $\vvvert x\vvvert\geq\|i(x)\|\geq\|a\|-\|a\|/4=3\|a\|/4$. Then
\[\|a-i(x)\|\geq\|\pi_u(a-i(x))\|\geq\|\pi_u(i(x))\|-\|\pi_u(a)\|\geq\|\pi_x(i(x))\|-\|\pi_u(a)\|=\]
\[\|\rho_x(x)\|-\|\pi_u(a)\|=\vvvert x\vvvert-\|\pi_u(a)\|,\]
which implies that
\[\|\pi_u(a)\|\geq \vvvert x\vvvert - \|a-i(x)\|\geq 3\|a\|/4-\|a\|/4=\|a\|/2>0.\]
Hence $\pi_u(a)\neq 0$, and since $a\neq 0$ was arbitrary, $\pi_u$ is faithful.
\end{proof}

As we will see at the end of Section \ref{sec:cuntz}, contrary to what happens if we only have norm relations, for a representation $\pi:C^*\triple\to B(H)$, even if it is faithful, it is not always true that $\pi\circ i$ is a representation of $\triple$.

\section{Examples}\label{sec:Examples}

Although some examples fall in a broader class of algebras, we decided to use some simpler examples to give all the details of the construction. The main tool will be interpreting an infinite sum as a net that converges to zero in the strong operator topology.

As it is implicitly done in \cite{MR813640}, equality relations involving only finite sums and finite products can be written as a pair $(x,0)$ for some $x\in\falg$. For instance, if you want to say that a certain generator $S$ is a partial isometry, that is $SS^*S=S$, we can give the relation $(SS^*S-S,0)$.

The following lemmas will be useful to study our examples.

\begin{lemma}\label{lem:positives.converging.to.zero}
If $(P_i)_{i\in I}$ is a net of positive operators on a Hilbert space $H$ such that $P_i\leq P_j$ if $i\leq j$ and $s-\lim_{i\to\infty}P_i=0$, then $P_i=0$ for all $i\in I$.
\end{lemma}
\begin{proof}
Let $\xi\in H$ be given and notice that $(\langle P_i\xi,\xi\rangle)_{i\in I}$ is a non-decreasing net of non-negative real numbers that converges to zero. This implies that $\langle P_i\xi,\xi\rangle = 0$ for all $i\in I$. Since $\xi$ is arbitrary and for each $i\in I$, $P_i$ is a positive operator, then $P_i=0$ for all $i\in I$.
\end{proof}
\begin{lemma}\label{lem:infinite.sum.projections}
Let $(P_i)_{i\in I}$ be a net of projections on a Hilbert space $H$ such that $\sum_{i\in I}P_i=1$ with respect to strong operator topology. Then $P_i P_j=0$ whenever $i\neq j$.
\end{lemma}
\begin{proof}
Let $\mathcal{P}_0(I)$ the set of finite subsets of $I$. Translating the infinite sum $\sum_{i\in I}P_i=1$ to a net, we obtain the net $(\sum_{i\in X}P_i)_{X\in\mathcal{P}_0(I)}$ which converges to 1 with respect to strong operator topology. Fix $i_0\in I$ and observe that the subnet $(\sum_{i\in X}P_i)_{i_0\in X \in\mathcal{P}_0(I)}$ also converges to 1 with respect to strong operator topology. By multiplying by $P_{i_0}$ on the left, we conclude that the net $(\sum_{i\in X\setminus\{i_0\}}P_{i_0}P_i)_{i_0\in X \in\mathcal{P}_0(I)}$ is non-decreasing, of positive operators and converges to $0$ with respect to strong operator topology. By Lemma \ref{lem:positives.converging.to.zero}, $\sum_{i\in X\setminus\{i_0\}}P_{i_0}P_i=0$ for every $X\in\mathcal{P}_0(I)$ such that $i_0\in X$. Fix $j_0\in I$ such that $j_0\neq i_0$ and consider $X=\{i_0,j_0\}$. For this choice of $X$, we obtain $P_{i_0}P_{j_0}=0$. Since $i_0$ and $j_0$ are arbitrary, the result follows.
\end{proof}

\begin{lemma}\label{lem:conv.zero.basis}
Let $H$ be a Hilbert space with orthonormal basis $\{e_{\lambda}\}_{\lambda\in\Lambda}$ and let $\{T_i\}_{i\in I}$ be a bounded net in $B(H)$. If $\lim_{i\to\infty}T_i(e_{\lambda})=0$ for all $\lambda\in\Lambda$, then $s-\lim_{i\to\infty}T_i=0$.
\end{lemma}

\begin{proof}
By linearity and the continuity of operation on $H$, we can see that for every $h\in\spanop \{e_{\lambda}\}_{\lambda\in\Lambda}$, we have that $\lim_{i\to\infty}T_i(e_{\lambda})=0$. Fix $M>0$ such that $\|T_i\|\leq M$ for all $i\in I$ and let $h\in H$ be arbitrary. Given $\epsilon>0$, there exists $h_0\in \spanop\{e_{\lambda}\}_{\lambda\in\Lambda}$ such that $\|h-h_0\|\leq \epsilon/2M$. Also, there exists $i_0$ such that for all $i\geq i_0$, $\|T_i(h_0)\|<\epsilon/2$. Then, for all $i\geq i_0$, we have that
\[\|T_i(h)\|=\|(T_i(h-h_0)+T_i(h_0)\|\leq \|T_i\|\|h-h_0\|+\|T_i(h_0)\|<M\frac{\epsilon}{2M}+\frac{\epsilon}{2}=\epsilon.\]
The result follows.
\end{proof}

\subsection{Cuntz algebras}\label{sec:cuntz}
The class of algebras $\cuntz{n}$ studied by Cuntz \cite{MR467330}, when $2\leq n <\infty$, can be seen as the universal unital C*-algebra generated by $n$ isometries $S_1,\ldots,S_n$ such that
\begin{equation}\label{eq:cuntz}
    \sum_{i=1}^n S_iS_i^*=1.
\end{equation}
For $n=\infty$, Cuntz defined $\cuntz{\infty}$ as a C*-subalgebra of $B(H)$ generated by isometries $\{S_i\}_{i=1}^{\infty}$ on an infinite dimensional Hilbert space $H$ such that $\sum_{i=1}^n S_iS_i^*\leq 1$ for all $n$. He showed this algebra is simple so it does not depend on $H$ or the isometries. It known that we can see $\cuntz{\infty}$ as the universal unital C*-algebra generated by isometries $\{S_i\}_{i=1}^{\infty}$ such that $S_i^*S_j=0$, whenever $i\neq j$.

We would like to use \eqref{eq:cuntz} as a relation when $n=\infty$, however this cannot be true with respect to the norm topology, since in $\cuntz{\infty}$, for all $r$, we have that $1-\sum_{i=1}^r S_iS_i^*$ is a projection of norm 1. 

We can interpret \eqref{eq:cuntz} as a sum with respect to the strong operator topology. More specifically, we consider the triple $\triple$, where $\Gg=\{T_i\}_{i=1}^{\infty}\dot{\cup}\{1\}$, $\rel_N$ is the set of relations that say that $1$ is a unit and $T_i^*T_i=1$ for all $i$, and $\rel_S$ is the singleton set with element $(1-\sum_{i=1}^r T_iT_i^*)_{r=1}^{\infty}$. To see that the triple is admissible, we take a representation $\rho$ of $\triple$ and notice that the set $\{\rho(T_i)\}_{i=1}^{\infty}$ is composed of isometries so that $\rho(T_i)\rho(T_i)^*$ are projections such that $\sum_{i=1}^
{\infty}\rho(T_i)\rho(T_i)^*=1$ with respect to the strong operator topology. By Lemma \ref{lem:infinite.sum.projections}, the projections $\rho(T_i)\rho(T_i)^*$ are mutually orthogonal. This implies that $1-\sum_{i=1}^
{r}\rho(T_i)\rho(T_i)^*$ is a projection for all $r$, so its norm is bounded by $1$. By Proposition \ref{prop:admissibility}, $\triple$ is admissible.

To prove that $\cuntz{\infty}\cong C^*\triple$, notice that if $\rho_u$ is the faithful representation given by Theorem \ref{thm:faithful.rep}, then the above computations show that $\rho_u(T_i)^*\rho_u(T_j)=0$ if $i\neq j$, and hence $T_i^*T_j=0$ for $i\neq j$ inside $C^*\triple$. By the universal property of $\cuntz{\infty}$, there is a unital *-homomorphism $\varphi:\cuntz{\infty}\to C^*\triple$ such that $\varphi(S_i)=T_i$ for all $i$. We build an inverse for $\varphi$ using Theorem \ref{thm:universal.property}. Let $H$ be an infinite dimensional separable Hilbert space and decompose $H$ as $H=\bigoplus_{n=1}^{\infty}H_n$, where each $H_n$ is again an infinite dimensional separable Hilbert space. For each $i\in\nn$ with $i\geq 1$, let $R_i$ be an isometry from $H$ to $H_i$ seen as an element of $B(H)$. Clearly $\sum_{i=1}^n R_iR_i^*\leq 1$ for every $n\geq 1$, so that there exists a faithful (because $\cuntz{\infty}$ is simple) representation $\pi:\cuntz{\infty}\to B(H)$ such that $\pi(S_i)=R_i$. Also notice that $\sum_{i=1}^{\infty}R_iR_i^*=1$ with respect to the strong operator topology. By Remark \ref{rem:hom.from.generators}, there is a *-homomorphism $\psi:\falg\to\cuntz{\infty}$ such that $\psi(T_i)=S_i$. We then get that $\pi\circ\psi$ is a representation of $\triple$. The map $\widetilde{\psi}:C^*\triple\to\cuntz{\infty}$ given by Theorem \ref{thm:universal.property} is a unital *-homomorphism such that $\widetilde{\psi}(T_i)=S_i$ for all $i$, which then implies that $\widetilde{\psi}=\varphi^{-1}$.

Another way to prove that $\varphi$ is a *-isomorphism would be to observe that $\varphi$ is surjective, and since $\cuntz{\infty}$ is simple, it would suffice to show that $C^*\triple$ is not the algebra $\{0\}$.

We now show that there exists a faithful representation $\rho$ of $\cuntz{\infty}$ such that $\rho\circ\psi$ is not a representation of $\triple$. The construction is analogous to the construction of $\pi$ above. The difference is that we add an extra non-zero subspace $H_0$, in the direct sum, that is, $H=\bigoplus_{n=0}^{\infty}H_n$. Defining $R_i$ as above, we get a faithful representation $\rho:\cuntz{\infty}\to B(H)$. In this case, notice that $\sum_{i=1}^{\infty}R_iR_i^*$ converges with respect to the strong operator topology to the projection on the proper subspace $\bigoplus_{n=1}^{\infty}H_n$, and therefore $\rho\circ\psi$ is not a representation of $\triple$.

\subsection{Exel-Laca algebras} In \cite{MR561974}, Cuntz and Krieger defined an algebra $\cuntz{A}$ associated with a finite square matrix of zeros and ones $A$ and studied its relationship with topological Markov chains. They remarked that a similar construction could be done with infinite matrices as long as infinite sums are considered in the strong operator topology. In their paper, they defined the algebra already represented on a Hilbert space and showed that under some conditions on the matrix, the algebra is unique in the sense that it does not depend on the representation \cite[Theorem 2.13]{MR561974}.

When the matrix is finite, Blackadar described $\cuntz{A}$ as universal C*-algebra \cite[Example 1.3(e)(9)]{MR813640}. One of the main goals of \cite{MR1703078} was to define a universal C*-algebra of an infinite matrix of zeros and ones. Since they relied on Blackadar's construction, they couldn't use the same relations as Blackadar used since they would be dealing with infinite sums. They, then, had to find enough norm relations to describe the algebra in question. In this section, using the theory developed in Section \ref{sec:definition.universal.algebra}, we show that the algebra defined by Exel and Laca is indeed the same one obtained by allowing infinite sum relations in Blackadar's example.

\begin{definition}\cite[Definition 7.1]{MR1703078}
Given a set of indices $I$ and a 0-1 matrix $A=(A_{ij})_{i,j\in I}$ with no identically zero rows, the unital Exel-Laca algebra $\elu$ is the universal C*-algebra generated by a family of partial isometries $\{S_i\}_{i\in I}$ and a unit $1$ satisfying the following relations:
\begin{itemize}
    \item[EL1.] $S_i^*S_iS_j^*S_j=S_j^*S_jS_i^*S_i$, for all $i,j\in I$;
    \item[EL2.] $S_iS_i^*S_jS_j^*=0$, whenever $i\neq j$;
    \item[EL3.] $S_i^*S_iS_jS_j^*=A_{ij}S_jS_j^*$ for all $i,j\in I$;
    \item[EL4.] for all $X,Y\scj I$ finite such that
    \[A(X,Y,j):=\prod_{x\in X}A_{xj}\prod_{y\in Y}(1-A_{yj})\]
    is zero for all but a finite number of $j$'s, we have that,
    \[\prod_{x\in X}S_x^*S_x\prod_{y\in Y}(1-S_y^*S_y)=\sum_{j\in I} A(X,Y,j)S_jS_j^*.\]
\end{itemize}
\end{definition}

Let $I$ and $A$ be as in the above definition. Our next goal is to describe $\elu$ using infinite sum relations. As generators we set $\Gg=\{T_i\}_{i\in I}\dot{\cup}\{1\}$. The norm relations $\rel_N$ are the ones that say that $T_i$ is a partial isometry for each $i\in I$ and that $1$ is a unit. The SOT relations $\rel_S$ are given by:
\begin{itemize}
    \item[CK1.] $\sum_{i\in I} T_iT_i^*=1$;
    \item[CK2.] for each $i\in I$, $T_i^*T_i=\sum_{j\in I}A_{ij}T_jT_j^*$,
\end{itemize}
seen as nets in the natural way. This way, we get a generating triple $\triple$. To prove that this triple is admissible, we use Proposition \ref{prop:admissibility} and find constants $\eta_{\mathbf{x}}$ that bound the relations CK1 and CK2. As in the example of the Cuntz algebra, we see that the first is bounded by $1$ and that $\{\rho(T_i)\rho(T_i)^*\}_{i\in I}$ is a set of mutually orthogonal projections for every representation $\rho$ of $\triple$. This implies that the relations on CK2 are all bounded by $2$.

To build a *-homomorphism from $\elu$ to $C^*\triple$, we use a faithful representation $\pi$ of $C^*\triple$ on a Hilbert space $H$ as in Theorem \ref{thm:faithful.rep}. In order to simplify the notation, we will use $T_i$ instead of $\pi(T_i)$ in the computations below. As mentioned in the previous paragraph, EL2 holds for the $T_i$'s. To prove the remaining relations, take $h\in H$. By CK2, for $i,j\in I$,
\[T_i^*T_iT_j^*T_j(h)=\lim_{X\to\infty}\lim_{Y\to\infty}\sum_{x\in X}\sum_{y\in Y}A_{ix}A_{jy}T_xT_x^*T_yT_y^*(h)=\lim_{X\to\infty}\sum_{x\in X}A_{ix}A_{jx}T_xT_x^*(h),\]
where the limits are taken over the directed set of all finite subsets of $I$ and the last equality holds by EL2. Since $A_{ix}$ and $A_{jx}$ are numbers, we see that the same final expression is valid for $T_j^*T_jT_i^*T_i(h)$. It follows that $T_i^*T_iT_j^*T_j=T_j^*T_jT_i^*T_i$, which is EL1. Next, we prove EL3. For $i,j\in I$,
\[T_i^*T_iT_jT_j^*(h)=\lim_{X\to\infty} \sum_{x\in X} A_{ix}T_xT_x^*T_jT_j^*(h)=A_{ij}T_jT_j^*,\]
where the first equality is due to CK2 and the last equality follows from EL2. Finally, for EL4, let $X$, $Y$ and $J$ be finite subsets of $I$ such that $A(X,Y,j)=0$ for all $j\in I\setminus J$, and $A(X,Y,j)\neq 0$ for all $j\in J$. A similar argument as above shows that
\[\prod_{x\in X}T_x^*T_x(h)=\lim_{R\to\infty}\sum_{r\in R}\left(\prod_{x\in X}A_{xr}\right) T_rT_r^*(h),\]
and also, using CK1, that
\[\prod_{y\in Y}(1-T_y^*T_y)(h)=\lim_{S\to\infty}\sum_{s\in S}\left(\prod_{y\in Y}(1-A_{ys})\right) T_sT_s^*(h).\]
Then, by CK2 and EL2, we have that
\[\prod_{x\in X}T_x^*T_x\prod_{y\in Y}(1-T_y^*T_y)(h)=\]
\[\lim_{R\to\infty}\lim_{S\to\infty}\sum_
{r\in R}\sum_{s\in S}\left(\prod_{x\in X}A_{xr}\right)\left(\prod_{y\in Y}(1-A_{ys})\right)T_rT_r^*T_sT_s^*(h)=\]
\[\lim_{R\to\infty}\sum_
{r\in R}\left(\prod_{x\in X}A_{xr}\right)\left(\prod_{y\in Y}(1-A_{yr})\right)T_rT_r^*(h)=\]
\[\lim_{R\to \infty}\sum_{r\in R}A(X,Y,r)T_rT_r^*(h)=\sum_{j\in J}A(X,Y,j)T_jT_j^*(h)=\sum_{j\in I}A(X,Y,j)T_jT_j^*(h).\]

By the universal property of $\elu$, we have a unital *-homomorphism $\varphi:\elu\to C^*\triple$ such that $\varphi(S_i)=T_i$.

To find the inverse, we recall a faithful representation of $\elu$ given in \cite[Section~9]{MR1703078}. We define $P_A=\{(i_n)_{n\in\nn}\in I^{\nn}\mid A_{i_ni_{n+1}}=1,\text{ for all }n\in\nn\}$. The hypothesis that $A$ has no identically zero rows implies that $P_A$ is non-empty. Considering $\ell^2(P_A)$ be the Hilbert space with its canonical basis $\{\varepsilon_{\mu}\}_{\mu\in P_A}$, we define the operators $L_i:\ell^2(P_A)\to\ell^2(P_A)$ for each $i\in I$ by
\[L_i(\varepsilon_{\mu})=\begin{cases}
\varepsilon_{i\mu} & \text{if }A_{i\mu_0}=1, \\
0 & \text{if }A_{i\mu_0}=0,
\end{cases}
\]
where $\mu=\mu_0\mu_1\ldots\in P_A$, and notice that
\[L_i^*\varepsilon_{\mu}=\begin{cases}
\varepsilon_{\mu_1\mu_2\ldots} & \text{if }\mu_0=i, \\
0 & \text{if }\mu_0\neq i.
\end{cases}\]
Therefore, $L_iL_i^*$ is the projection on the space $\overline{\spanop}\{\varepsilon_{\mu}\mid \mu_0=i\}$, and $L_i^*L_i$ is the projection on the space $\overline{\spanop}\{\varepsilon_{\mu}\mid A_{i\mu_0}=1\}$.

Also let $\fr$ be the free group generated by $I$, with regular left representation $\lambda$ of $\fr$ on $\ell^2(\fr)$, whose canonical basis is denoted by $\{\delta_g\}_{g\in\fr}$. By \cite[Proposition 9.1]{MR1703078}, there is a unique faithful representation $\rho:\elu\to B(\ell^2(P_A)\otimes \ell^2(\fr))$ such that $\rho(S_i)=L_i\otimes\lambda_i$ for all $i\in I$.

We want to use Theorem \ref{thm:universal.property} in order to find a *-homomorphism from $C^*\triple$ to $\elu$. For that, it suffices that $\{L_i\otimes\lambda_i\}_{i\in I}\cup \{Id\}$ forms a representation of $\triple$ (see Remark \ref{rem:hom.from.generators}). The norm relations are satisfied since $\rho$ is a representation of $\elu$. We check the SOT relations. For that take $\mu\in P_A$ and $g\in\fr$, and notice that
\[(L_i\otimes\lambda_i)(L_i\otimes\lambda_i)^*(\varepsilon_{\mu}\otimes\delta_g)=L_iL_i^*(\varepsilon_{\mu})\otimes \delta_g,\]
and that
\[(L_i\otimes\lambda_i)^*(L_i\otimes\lambda_i)(\varepsilon_{\mu}\otimes\delta_g)=L_i^*L_i(\varepsilon_{\mu})\otimes \delta_g.\]

Since there is a unique $i\in I$ such that $\mu_0=i$, we have that
\[\sum_{i\in I}(L_i\otimes\lambda_i)(L_i\otimes\lambda_i)^*(\varepsilon_{\mu}\otimes\delta_g)=\varepsilon_{\mu}\otimes\delta_g\]
and, by Lemma \ref{lem:conv.zero.basis}, CK1 follows.

Noticing that $\overline{\spanop}\{\varepsilon_{\mu}\mid A_{i\mu_0}=1\}=\bigoplus_{j:A_{ij}=1}\overline{\spanop}\{\varepsilon_{\mu}\mid \mu=j\}$, we conclude that
\[(L_i\otimes\lambda_i)^*(L_i\otimes\lambda_i)(\varepsilon_{\mu}\otimes\delta_g)=\sum_{j:A_{ij}=1}L_jL_j^*(\varepsilon_{\mu})\otimes \delta_g=\sum_{j\in I}A_{ij}(L_j\otimes\lambda_j)(L_j\otimes\lambda_j)^*(\varepsilon_{\mu}\otimes\delta_g),\]
and hence, by Lemma \ref{lem:conv.zero.basis}, CK2 holds.

By Theorem \ref{thm:universal.property}, there exists an unital *-homomorphism $\psi:C^*\triple\to\elu$ such that $\psi(T_i)=S_i$. By the universality of these algebras, we have then that $\psi=\varphi^{-1}$.

We now consider the not necessarily unital version of Exel-Laca algebras.

\begin{definition}\cite[Definition 8.1]{MR1703078}
Given a set of indices $I$ and a 0-1 matrix $A=(A_{ij})_{i,j\in I}$ with no identically zero rows, the Exel-Laca algebra $\cuntz{A}$ is the C*-subalgebra of $\elu$ generated by $\{S_i\}_{i\in I}$.
\end{definition}

Notice that the above definition for $\cuntz{A}$ is not as a universal C*-algebra. In order to do that, we have to observe that the relation EL4 uses a unit, but $\cuntz{A}$ is not necessarily unital. We can rewrite EL4 without the use unit whenever $X\neq\emptyset$. In fact, Exel and Laca showed in \cite[Proposition 8.5]{MR1703078} that there exists $Y\scj I$ finite such that $A(\emptyset,Y,j)=0$ for all but a finitely many $j$'s if and only if $\cuntz{A}$ is unital, in which case $\cuntz{A}=\elu$. So we have to consider two cases in order to describe $\cuntz{A}$ as a universal C*-algebra. If there exists $Y\scj I$ finite such that $A(\emptyset,Y,j)=0$ for all but a finitely many $j$'s, then $\cuntz{A}$ is the unital C*-algebra generated by a family of partial isometries $\{S_i\}_{i\in I}$ satisfying EL1-EL4. If there is no such $Y$, then we rewrite EL4 without the use of unit and $\cuntz{A}$ is the universal C*-algebra generated by a family of partial isometries $\{S_i\}_{i\in I}$ satisfying E1-E4.

By allowing infinite sums we are able to describe $\cuntz{A}$ as universal C*-algebra in a unified way. We just need to change CK1 in order to avoid the need for a unit. If we check our computations as well as the example of the Cuntz algebra, we see that we can replace CK1 by the relations $T_i^*T_j=0$, whenever $i\neq j$. Similarly to what we have done for $\elu$ above, we can show that $\cuntz{A}$ is isomorphic to the universal C*-algebra generated by a family of partial isometries $\{T_i\}_{i\in I}$ with mutually orthogonal final projections such that
\begin{itemize}
    \item[CK.] for each $i\in I$, $T_i^*T_i=\sum_{j\in I}A_{ij}T_jT_j^*$.
\end{itemize}

In the case that there exists $Y\scj I$ finite such that $A(\emptyset,Y,j)=0$ for all but a finitely many $j$'s, we should be able to find a unit in our version of $\cuntz{A}$. Indeed, supposing the existence of such $Y$, consider the sets $J_1=\{j\in I\mid A(\emptyset,Y,j)=0\}$ and $J_2=\{j\in I\mid A(\emptyset,Y,j)\neq 0\}$, so that $J_2$ is finite. Notice that $A(\emptyset,Y,j)=0$ if and only if there exists $y\in Y$ such that $A_{yj}=1$. We consider a faithful representation $\rho_u$ of $\cuntz{A}$ as in Theorem \ref{thm:faithful.rep}, and to simplify the notation we use the isomorphism $\cuntz{A}\cong \rho_u(\cuntz{A})$ as a equality. Using the inclusion-exclusion principle and CK, we see that, with respect to the strong operator topology,
\[\sum_{\emptyset\neq Z\scj Y}(-1)^{|Z|+1}\prod_{z\in Z}T_z^*T_z=\sum_{j\in J_1}T_jT_j^*.\]
Notice that the left hand side is an element of $\cuntz{A}$, and if we define
\[U=\sum_{\emptyset\neq Z\scj Y}(-1)^{|Z|+1}\prod_{z\in Z}T_z^*T_z+\sum_{j\in J_2}T_jT_j^*,\]
then $U\in\cuntz{A}$ and $UT_i=T_i=T_iU$ for all $i\in I$, and hence $U$ is a unit for $\cuntz{A}$.

\section{C*-algebras generated by projections and partial isometries}\label{sec:algebra.partial.isometry}

In this section, we prove that under some conditions a C*-algebra generated by projections and partial isometries can be described as a universal C*-algebra using only norm relations and generators corresponding to the projections and partial isometries. This means that we can define a C*-algebra using SOT relations or any other means, and if we prove the conditions of Theorem \ref{lem:commuting.projections.universality} or Theorem \ref{thm:alg.gen.proj.part.iso}, we know that there is a set of norm relations describing the same C*-algebra.

We start by studying commutative C*-algebras generated by projections seen as themselves or as a C*-subalgebra of a not necessarily commutative C*-algebra.

\begin{lemma}\label{lem:set.of.orthogonal.projections}
Let $A$ be a C*-algebra and $Q=\{q_1,q_2,\ldots,q_n\}\subseteq A$ be a subset of mutually commuting projections. Then the set
$$P=\left\{\prod_{j\in X}q_j\prod_{k\in \{1,\ldots,n\}\backslash X}(1-q_k)\,\biggl |\, \emptyset \neq X\subseteq \{1,\ldots,n\}\right\}$$
is formed by projections that are mutually orthogonal and $\spanop P$ is the C*-subalgebra of $A$ generated by $Q$. 
\end{lemma}

\begin{proof}
Before starting the proof, we note that the elements in $P$ can be written without using the unit because $X\neq\emptyset$. Clearly, the elements in $P$ are projections. Let $X_1$ and $X_2$ be non-empty subsets of $\{1,\ldots,n\}$ such that $X_1\neq X_2$. Without loss of generality, consider $m\in\{1,\ldots,n\}$ such that $m\in X_1$ e $m\notin X_2$. Then the projection 
$$\prod_{j\in X_1}q_j\prod_{k\in \{1,\ldots,n\}\backslash X_1}(1-q_k)$$
contains the factor $q_m$ and the projection
$$\prod_{j\in X_2}q_j\prod_{k\in \{1,\ldots,n\}\backslash X_2}(1-q_k)$$
contains the factor $1-q_m$. Since the projections $q_i$ commute with each other, then 
$$\left(\prod_{j\in X_1}q_j\prod_{k\in \{1,\ldots,n\}\backslash X_1}(1-q_k)\right)\left(\prod_{j\in X_2}q_j\prod_{k\in \{1,\ldots,n\}\backslash X_2}(1-q_k)\right)=0,$$
showing that the projections in $P$ are mutually orthogonal. From this, it is easy to see that $\spanop\,P$ is a C*-subalgebra of $A$, since it is closed under the operations and it is finite dimensional. Furthermore, for every $i\in\{1,\ldots,n\}$,
$$q_i=\sum_{\stackrel{X\subseteq\{1,\ldots,n\}}{i\in X}}\left(\prod_{j\in X}q_j\prod_{k\in \{1,\ldots,n\}\backslash X}(1-q_k)\right) \in \spanop P,$$
which says that $C^*(Q)\subseteq \spanop P$. Since that it is clear that $\spanop P\subseteq C^*(Q)$, the result follows.
\end{proof}

\begin{theorem}\label{lem:commuting.projections.universality}
Let $A$ be a commutative C*-algebra which is generated by a family of (mutually commuting) projections $\{q_i\}_{i\in I}$. Denote by $\rel_0$ the set of all finite algebraic relations in $A$ involving the generators in $\{q_i\}_{i\in I}$, i.e., $\rel_0$ is the set of all relations of the form
$$\sum_m^{\mathrm{finite}}\lambda_m\prod_n^{\mathrm{finite}}q_{i_{mn}}=0$$
that are satisfied in $A$. Consider a set $\{p_i\}_{i\in I}$ with the same cardinality of $I$ and define
$$\rel_1 = \{p_i=p_i^*=p_i^2\,|\,i\in I\}\cup\{p_{i_1}p_{i_2}=p_{i_2}p_{i_1}\,|\,i_1,i_2\in I\}$$
and
$$\rel_2=\left\{\sum_m\lambda_m\prod_np_{i_{mn}}=0 \,\biggl|\, \sum_m\lambda_m\prod_nq_{i_{mn}}=0 \in \rel_0\right\}.$$
Then $C^*(\{p_i\}_{i\in I}, \rel_1\cup\rel_2)$ is *-isomorphic to $A$ by the map $p_i\mapsto q_i$.
\end{theorem}

\begin{proof}
Since the relations in $\rel_1$ make each $p_i$ a projection, then the pair $(\{p_i\}_{i\in I}, \rel_1\cup\rel_2)$ is admissible and, since relations in $\rel_1\cup\rel_2$ are satisfied in $A$ (replacing $p_i$ by $q_i$), then there is a *-homomorphism $\phi:C^*(\{p_i\}_{i\in I}, \rel_1\cup\rel_2)\to A$ given by $\phi(p_i)=q_i$. Denote by $A_0$ the $*$-subalgebra of $A$ generated by $\{q_i\}_{i\in I}$, i.e., $A_0$ is the $\cn$-span of finite products of $q_i$'s, and consider the map $\psi_0: A_0\to C^*(\{p_i\}_{i\in I}, \rel_1\cup\rel_2)$ given by $\psi_0(q_i)=p_i$ in the generators and extended to $A_0$ in such way that it is a *-homomorphism. We need to verify that $\psi_0$ is well defined. To see this, consider $x\in A_0$ such that
$$x=\underbrace{\sum_j\lambda_j\prod_kq_{i_{jk}}}_{x_1}=\underbrace{\sum_u\mu_u\prod_vq_{i_{uv}}}_{x_2}.$$
Since $x_1-x_2=0\in\rel_0$, then
$$\sum_j\lambda_j\prod_kp_{i_{jk}} - \sum_u\mu_u\prod_vp_{i_{uv}}=0 \in \rel_2$$
and, hence, $\psi_0(x_1)=\psi_0(x_2)$, showing $\psi_0$ is well defined. We claim that $\psi_0$ is contractive. Indeed, let 
$$y=\sum_{j=1}^J\mu_j\prod_{k=1}^{K_j}q_{i_{jk}}\in A_0.$$
Denote by $\tilde{A}_0$ the *-subalgebra of $A_0$ generated by the (finite) set $\{q_{i_{jk}}\}_{j,k=1}^{J,K_j}$. By Lemma \ref{lem:set.of.orthogonal.projections}, $\tilde{A}_0$ is a C*-algebra and so $\psi_0$ restricted to $\tilde{A}_0$ is contractive. Since $y\in\tilde{A}_0$, then $||\psi_0(y)||\leq||y||$ and since $y\in A_0$ is arbitrary, then $\psi_0$ is contractive. Denote by $\psi$ the extension of $\psi_0$ to the closure of $A_0$, which is $A$. Clearly, $\phi$ and $\psi$ are the inverse of each other.
\end{proof}

Next corollary gives a step in looking for a universal version of a C*-algebra generated by projections and partial isometries. In some examples in the theory such as for graph C*-algebras \cite{MR3119197} and labelled spaces C*-algebras \cite{MR3680957}, there is a commutative C*-subalgebra generated by projections that plays an important role when studying the structure of the algebra.

\begin{corollary}\label{cor:subalgebra.isomorphism}
Let $A$ be a C*-algebra generated by a family of partial isometries $\{t_j\}_{j\in J}$ and $B$ a commutative C*-subalgebra of $A$ generated by family of projections $\{q_i\}_{i\in I}$. Consider a set $\{s_j\}_{j\in J}$ with the same cardinality of $J$, a set $\{p_i\}_{i\in I}$ with the same cardinality of $I$ and disjoint from $\{s_j\}_{j\in J}$ and define $\Gg=\{p_i\}_{i\in I}\cup \{s_j\}_{j\in J}$. Let $\rel_0$ be the set of all finite algebraic relations in $A$ involving $t_j$ and $q_i$ and denote by $\rel_1$ the set of all relations obtained from $\rel_0$ by replacing each $t_j$ and each $q_i$ with $s_j$ and $p_i$, respectively. Then, the natural surjective *-homomorphism $\Phi:C^*(\Gg, \rel_1)\to A$ restricts to an isomorphism $\widetilde{\Phi}:C^*(\{p_i\}_{i\in I})\to B$, where $C^*(\{p_i\}_{i\in I})$ is the C*-subalgebra of $C^*(\Gg, \rel_1)$ generated by the family 
$\{p_i\}_{i\in I}$.
\end{corollary}

\begin{proof}
Clearly $\widetilde{\Phi}$ is surjective, so we only need to proof the injectivity. Let $\{r_i\}_{i\in I}$ be a set with the same cardinality of $I$ and $\rel_2$ be the set of all relations obtained by considering in $\rel_0$ those relations that can be written using only $\{q_i\}_{i\in I}$ and then replacing each $q_i$ with $r_i$. By Theorem \ref{lem:commuting.projections.universality}, the universal *-homomorphism $\Psi:C^*(\{r_i\}_{i\in I}, \rel_2)\to B$ given by $\Psi(r_i)=q_i$ is an isomorphism. The universal property of $C^*(\{r_i\}_{i\in I}, \rel_2)$ also give us a surjective *-homomorphism $\Theta:C^*(\{r_i\}_{i\in I}, \rel_2)\to C^*(\{p_i\}_{i\in I})$ given by $\Theta(r_i)=p_i$. Since $(\Theta\circ\Psi^{-1})\circ \widetilde{\Phi}(p_i)=p_i$ and $\widetilde{\Phi}\circ(\Theta\circ\Psi^{-1})(q_i)=q_i$, then $\widetilde{\Phi}:C^*(\{p_i\}_{i\in I})\to B$ is an isomorphism.
\end{proof}

The next result is inspired by what in the literature is known as the Gauge Invariance Theorem. This class of results can be thought as being a uniqueness theorem and describes conditions to conclude that a map defined in a universal C*-algebra is injective. As we have seen in the previous results, if we use the generators of a C*-algebra to define a universal version of it, we automatically get a surjective *-homomorphism and, under conditions of Corollary \ref{cor:subalgebra.isomorphism}, we get a *-isomorphism of a distinguished C*-subalgebra. In the next theorem, under stronger hypothesis, we obtain a *-isomorphism for the whole algebra.

\begin{notation}
In next theorem, we use $\left< B\right>$ to denote the set of all finite products of elements in $B$.
\end{notation}

\begin{theorem}\label{thm:alg.gen.proj.part.iso}
Let $A$ be a C*-algebra generated by a family of projections $\{q_i\}_{i\in I}$ and a family of partial isometries $\{t_j\}_{j\in J}$ and suppose that there exists a strongly continuous action $\gamma:\circulo\to \mathrm{Aut}(A)$ such that for every $z\in\circulo$, $i\in I$ and $j\in J$, we have that $\gamma_z(q_i)=q_i$ and $\gamma_z(t_j)=zt_j$. Let $N=\left<\{q_i\}_{i\in I}\cup\{t_j\}_{j\in J}\cup\{t_j^*\}_{j\in J}\right>\cap A^{\gamma}$ and suppose that for every finite subset $Y\scj N$, the C*-algebra generated by $Y$ is finite dimensional. Consider a set $\{p_i\}_{i\in I}$ with the same cardinality of $I$ and a set $\{s_j\}_{j\in J}$ with the same cardinality of $J$ and disjoint from $\{p_i\}_{i\in I}$. Let $\Gg=\{p_i\}_{i\in I}\cup \{s_j\}_{j\in J}$ and $\Phi:\falg\to A$ be the map given by $\Phi(p_i)=q_i$ for all $i\in I$ and $\Phi(s_j)=t_j$ for all $j\in J$. Define $M=\{x\in \left<\Gg\cup\Gg^*\right> \, | \, \Phi(x)\in N\}$ and define sets of norm relations
\[\rel_1=\{p_i=p_i^*=p_i^2\mid i\in I\}\cup\{s_js_j^*s_j=s_j\mid j\in J\}\]
and
\[\rel_2=\left\{\sum_{x\in M}^{\text{finite}}\lambda_x x=0 \, \ \biggl| \, \ \Phi\left(\sum_{x\in M}^{\text{finite}}\lambda_x x\right)=0\right\}.\]
Then $\Phi$ factors to an isomorphism between $C^*(\Gg,\rel_1\cup\rel_2)$ and $A$.
\end{theorem}

\begin{proof}
The proof will follow the steps: we define a strong continuous action $\beta:\circulo\to\mathrm{Aut}(C^*(\Gg,\rel_1\cup\rel_2))$, find its invariant set $C^*(\Gg,\rel_1\cup\rel_2)^\beta$, show that $\Phi$ is injective over $C^*(\Gg,\rel_1\cup\rel_2)^\beta$ and then conclude that $\Phi$ is an isomorphism.

Let $y\in N$ and let $m$, $n$ and $k$ be the number of $q_i$'s, $t_j$'s and $t_j^*$'s in $y$, respectively. For $z\in\circulo$, we have $\gamma_z(y) = z^{n-k}y$ and since $y\in A^\gamma$, then $n=k$. Then $x\in M$ if and only if $x$ has the same quantity of $s_j$'s and $s_j^*$'s. Fix $z\in\circulo$ and consider $\beta_z:\falg\to C^*(\Gg,\rel_1\cup\rel_2)$ given by $\beta_z(p_i)=p_i$ and $\beta_z(s_j)=zs_j$. It's clear that $\beta_z$ satisfies the relations in $\rel_1$. From the previous characterization of elements in $M$, if $x\in M$ then $\beta_z(x)=x$ and, therefore, $\beta_z$ satisfies the relations in $\rel_2$ too. Hence, $\beta_z$ factors to an endomorphism of $C^*(\Gg,\rel_1\cup\rel_2)$. Since $\beta_z\beta_w$ and $\beta_{zw}$ coincides in $\Gg$, then $\beta_z\beta_w=\beta_{zw}$ showing $\beta$ is an action by automorphisms. A standard $\varepsilon/3$ argument shows that $\beta$ is strongly continuous.

We claim that $C^*(\Gg,\rel_1\cup\rel_2)^\beta=C^*(M)$, where $C^*(M)$ is the C*-subalgebra of $C^*(\Gg,\rel_1\cup\rel_2)$ generated by $M$. Clearly, $C^*(M)\subseteq C^*(\Gg,\rel_1\cup\rel_2)^\beta$. To see the other inclusion, consider the conditional expectation $E:C^*(\Gg,\rel_1\cup\rel_2)\to C^*(\Gg,\rel_1\cup\rel_2)^\beta$ associated with $\beta$ given by
$$E(b) = \int_{\circulo}\beta_z(b)dz.$$
If $x\in \left<\Gg\cup\Gg^*\right>$, then $E(x)=x$ if $x\in M$ and $E(x)=0$ if $x\notin M$. Let $b\in C^*(\Gg,\rel_1\cup\rel_2)^\beta$ and choose a sequence $(b_k)_{k\in\nn}$ in $\falg$ converging to $b$. Each $b_k$ is a linear combination of elements in $\left<\Gg\cup\Gg^*\right>$ and then $E(b_k)$ is a linear combination of elements in $M$. Thus,
$$b = E(b) = E(\lim_k b_k) = \lim_k E(b_k) \in C^*(M).$$

Now, let's deal with $\Phi$. Since the image by $\Phi$ of relations in $\rel_1\cup\rel_2$ are obviously satisfied in $A$, then $\Phi$ factors to a surjective homomorphism $\Phi:C^*(\Gg,\rel_1\cup\rel_2)\to A$. We claim that $\Phi$ is injective on $C^*(\Gg,\rel_1\cup\rel_2)^\beta$. To achieve this goal, it suffices to show that $\Phi$ is isometric on linear combinations of elements of $M$, which is dense in $C^*(\Gg,\rel_1\cup\rel_2)^\beta$. By relations in $\rel_2$, $\Phi$ is injective on linear combinations of $M$. Now, fix $X$ a finite subset of $M$. By hypothesis, $C^*(\Phi(X))$ is finite dimensional and since $\Phi$ is injective on linear combinations of elements of $M$, then $C^*(X)$ is finite dimensional. Furthermore, since $C^*(X)$ is contained on the set of linear combinations of elements of $M$, then the restriction $\Phi:C^*(X)\to C^*(\Phi(X))$ is an injective homomorphism between C*-algebras and, hence, isometric. Since $X$ is arbitrary, then $\Phi$ is isometric on linear combinations of elements of $M$. This shows $\Phi$ is injective on
$C^*(\Gg,\rel_1\cup\rel_2)^\beta$.

Clearly $\Phi(\beta_z(b))=\gamma_z(\Phi(b))$ for all $b\in \falg$, and all $z\in\circulo$. By continuity we can extend to all $b\in C^*(\Gg,\rel_1\cup\rel_2)$. Hence,
\[\left\|\Phi\left(\int_{\circulo}\beta_z(b)dz\right)\right\|=\left\|\int_{\circulo}\gamma_z(\Phi(b))\right\|\leq \|\Phi(b)\|\]
for all $b\in C^*(\Gg,\rel_1\cup\rel_2)$. By \cite[Lemma~2.2]{MR1126190}, $\Phi$ is injective.
\end{proof}

\begin{example}[Ultragraph C*-algebras]
An ultragraph is a quadruple $\ug=\ugquad$, where $E^0$ and $\ug^1$ are sets, whose elements are called vertices and edges respectively, and $s:\ug^1\to E^0$ and $r:\ug^1\to\mathcal{P}(E^0)$ are functions called source and range, where $\mathcal{P}(E^0)$ is the power set of $E^0$. We denote by $\ug^0$ the smallest subset of $\mathcal{P}(E^0)$ closed under finite intersections and unions containing $\emptyset$, $r(e)$ for all $e\in \ug^1$ and $\{v\}$ for all $v\in E^0$. We adapt Tomforde's definition in \cite{MR2050134} allowing infinite sum relations. We define $C^*(\ug)$ to be the universal C*-algebra generated by a set of projections $\{p_A\}_{A\in\ug^0}$ and partial isometries with mutually orthogonal ranges $\{s_e\}_{e\in\ug^1}$ satisfying the relations:
\begin{enumerate}
    \item $p_{\emptyset}=0$, $p_{A\cap B}=p_Ap_B$ and $p_{A\cup B}=p_A+p_B-p_{A\cap B}$ for all $A,B\in\ug^0$;
    \item $s_e^*s_e=p_{r(e)}$ for all $e\in\ug^1$;
    \item $p_{\{v\}}=\sum_{e\in s^{-1}(v)}s_es_e^*$ for all $v\in E^0$ such that $s^{-1}(v)\neq\emptyset$.
\end{enumerate}
As in the examples of Section \ref{sec:Examples}, we can prove that the relations are admissible and that $C^*(\ug)$ coincides with the one defined by Tomforde in \cite{MR2050134} and, in fact, both examples of Section \ref{sec:Examples} can be modeled using ultragraphs. Assuming that we did not have Tomforde's description, let us prove that we can apply Theorem \ref{thm:alg.gen.proj.part.iso} to this example. Given $n\in\nn$ with $n\geq 1$, a path of length $n$ is a sequence of $n$ edges $\alpha=\alpha_1\ldots\alpha_n$ such that $s(\alpha_{i+1})\in r(\alpha_i)$ for all $i=1,\ldots,n-1$. The set of all paths of length $n$ is denoted by $\ug^n$. The elements of $\ug^0$ are thought to be paths of length zero and we define $\ug^*=\bigcup_{n=0}^{\infty}\ug^n$. For a path $\alpha=\alpha_1\ldots\alpha_n$ with positive length, we define $|\alpha|=n$ and $s_{\alpha}=s_{\alpha_1}\cdots s_{\alpha_n}$, and for $A\in \ug^0$, we set $|A|=0$ and $s_A=p_A$. As pointed out in \cite[Remark~2.10]{MR2050134}, any finite word in $\{p_A\}_{A\in\ug^0}$, $\{s_e\}_{e\in\ug^1}$ and $\{s_e^*\}_{e\in\ug^1}$ can be rewritten as $s_\alpha p_A s_\beta^*$ for some $\alpha,\beta\in\ug^*$ and $A\in\ug^0$. Also, it is standard to show that there is a strongly continuous action $\gamma$ of $\circulo$ on $C^*(\ug)$ given by $\gamma_z(p_A)=p_A$ and $\gamma_z(s_e)=zs_e$ for all $z\in\circulo$, $A\in\ug^0$ and $e\in\ug^1$.

If $M$ is the set described in the statement of Theorem \ref{thm:alg.gen.proj.part.iso}, we can see that elements of $M$ can be written as $s_\alpha p_A s_\beta^*$ for some $\alpha,\beta\in\ug^*$ and $A\in\ug^0$ where $|\alpha|=|\beta|$. A finite $X\scj M$ is of the form $\{s_{\alpha_1} p_{A_1} s_{\beta_1}^*,\ldots,s_{\alpha_n} p_{A_n} s_{\beta_n}^*\}$ for some $n\in\nn$ and $\alpha_i,\beta_i\in\ug^*$, $A_i\in\ug^0$ for all $i=1,\ldots,n$. We let $F$ be the set of all edges that appears in a path $\alpha_i$ or $\beta_i$ for some $i$, $N$ be the greatest length of all paths $\alpha_i$, and $Q=\{p_{A_i}\}_{i=1}^n\cup \{p_{r(e)}\}_{e\in F}$. Since the elements of $Q$ are commuting projections, we can apply Lemma \ref{lem:set.of.orthogonal.projections}, to find a set $P$ of mutually orthogonal projections such that $\spanop P=\spanop Q$. Consider also $W$ the set of all paths of length at most $N$ and whose edges belong to $F$ and $Y=\{s_\alpha p s_\beta^*\mid \alpha,\beta\in W, p\in P\}$. Using \cite[Lemmas 2.8 and 2.9]{MR2050134}, we see that $C^*(X)\scj C^*(Y)=\spanop Y$, concluding that $C^*(X)$ is finite dimensional, and therefore Theorem \ref{thm:alg.gen.proj.part.iso} can be applied, which means that the C*-algebra of an ultragraph can be defined using infinite sum relations but can be described using only norm relations.
\end{example}

\bibliographystyle{abbrv}
\bibliography{ref}

\end{document}